\documentclass[a4paper]{amsart}
\usepackage{amsmath,amsfonts,amssymb,amsthm,mathrsfs,xypic}
\usepackage{enumerate}
\usepackage{array}
\usepackage{graphicx}
\usepackage{longtable}
\usepackage{multirow}
\usepackage[left=3cm, right=3cm, top=3cm, bottom=3cm]{geometry}
\usepackage{pdflscape}
\usepackage{etex}
\usepackage{pictex}
\usepackage{ulem}
\usepackage[all,cmtip]{xy}
\usepackage{fancybox}
\usepackage{epic}
\usepackage[all]{xy}
\usepackage{tikz}
\usepackage{verbatim}

\def\Hom{\operatorname{Hom}}
\def\Ext{\operatorname{Ext}}
\def\Tor{\operatorname{Tor}}
\def\dim{\operatorname{dim}}

\def\proj{\operatorname{proj}}
\def\inj{\operatorname{inj}}

\newtheorem{thm}{Theorem}[section]

\newtheorem{lem}[thm]{Lemma}
\newtheorem{defn}[thm]{Definition}
\newtheorem{cor}[thm]{Corollary}
\newtheorem{prop}[thm]{Proposition}
\newtheorem{rem}[thm]{Remark}
\newtheorem{ex}[thm]{Example}

\usepackage{threeparttable}
\usepackage{extarrows}
\usepackage{mathrsfs}
\usepackage{chemarrow}
\usepackage{amscd}
\usepackage{amsthm}
\usepackage{latexsym,bm}
\xyoption{curve}
\usepackage{fancybox}
\usepackage{makecell}

\title[Exceptional cycles in triangular matrix algebras]
{Exceptional cycles in triangular matrix algebras}
\author[Peng Guo]{Peng Guo \\ \\
Department of mathematics, Tsinghua University\\ Beijing 100084, China}
\thanks{Supported by National Natural Science Foundation of China (NSFC Grant No. 11971255).}
\thanks{Email:\  guop@mail.tsinghua.edu.cn}

\begin{document}

\begin{abstract} An exceptional cycle in a triangulated category with Serre functor
is a generalization of a spherical object. Suppose that $A$ and $B$ are Gorenstein algebras, given a perfect exceptional $n$-cycle $E_*$ in $K^b(A\mbox{-}{\rm proj})$ and a perfect exceptional $m$-cycle $F_*$ in $K^b(B\mbox{-}{\rm proj})$, we construct an $A$-$B$-bimodule $N$, and prove the product $E_*\boxtimes F_*$ is an exceptional $(n+m-1)$-cycle in $K^b(\Lambda\mbox{-}{\rm proj})$, where $\Lambda=\begin{pmatrix}A & N\\ 0 & B \end{pmatrix}$. Using this construction, one gets many new exceptional cycles which is unknown before for certain class of algebras.

\vskip5pt

\end{abstract}

\maketitle

\vspace{-20pt}

\section{introduction}

Inspired by Kontsevich's homological mirror conjecture \cite{Ko} and the occurrence of certain braid actions in symplectic geometry, Seidel and Thomas
introduced the notion of spherical objects (they are mirror-symmetric analogues of Lagrangian spheres on a symplectic manifold, hence called spherical objects), and proved that it induced an auto-equivalence of the given triangulated category (see \cite{ST}).

Spherical objects were soon generalised to spherical functors (see \cite{R} and \cite{An}), and it was finally completed by Anno and Logvinenko for the dg setting \cite{AL}. An notable fact of spherical functors is that, the twist functor of a spherical functor is an auto-equivalence; conversely, any auto-equivalence of an (enhanced) triangulated category can be realised as the twist functor of a spherical functor \cite{Seg}. Spherical functors also appear in many areas, for example,
algebraic geometry (see \cite{Add} and \cite{D}), mathematical physics \cite{AA} and geometry invariant theory \cite{HLS}.

In a triangulated category $\mathcal T$ with Serre functor, Efimov \cite{E} and Broomhead, Pauksztello, Ploog \cite{BPP} introduced the notion of exceptional cycles (note that they were called spherical sequences in \cite{E}). Actually, exceptional $n$-cycles is a special case of spherical functors with the source triangulated category $D^b(k^n)$. Any exceptional cycle in $\mathcal T$ induces an auto-equivalence of $\mathcal T$, which is a generalization of tubular mutation \cite{M}. Given a triangulated category $\mathcal T$, it is always difficult to construct an auto-equivalence of $\mathcal T$. Thus it is important to determine all the exceptional cycle in $\mathcal T$. In a series of papers, all the exceptional cycles in several triangulated categories with Serre functor are classified, for example, the bounded derived category $D^b(kQ)$ of a finite acyclic quiver $Q$, the bounded homotopy category $K^b(A\mbox{-}{\rm proj})$ of a gentle algebra $A$, and certain triangulated categories arising from nilpotent representations (see \cite{GZ1}, \cite{GZ2}, \cite{GZ3}).

In this paper, we focus on constructing a new exceptional cycle from the old ones. Precisely, assume that $A$ and $B$ are Gorenstein algebras, given two exceptional cycles $(E_1, E_2, \cdots, E_n)$ and $(F_1, F_2, \cdots, F_m)$ in $K^b(A\mbox{-}{\rm proj})$ and $K^b(B\mbox{-}{\rm proj})$ respectively, we construct a new exceptional cycle in $K^b(\Lambda\mbox{-}{\rm proj})$, where $\Lambda$ is an upper triangular matrix algebra $\begin{pmatrix}A & N\\ 0&  B\end{pmatrix}$ with $N={_A}N_B$ an $A$-$B$-bimodule. In other words, given two auto-equivalences in $K^b(A\mbox{-}{\rm proj})$ and $K^b(B\mbox{-}{\rm proj})$ respectively, we construct an auto-equivalence of $K^b(\Lambda\mbox{-}{\rm proj})$ (since any exceptional cycle can induce an auto-equivalent in a given triangulated category).

\section{preliminary}

\subsection{Triangular matrix algebras}

Throughout, $k$ is an algebraically closed field, $A$ and $B$ are two finite dimensional $k$-algebras, $N={_A}N_B$ is an $A$-$B$-bimodule such that
$\begin{pmatrix}A & N\\ 0&  B\end{pmatrix}$ is a finite dimensional algebra, with addition and multiplication given by the ones of matrices. Denote by $A$-mod (resp. $B$-mod) the category of finitely generated left $A$-modules (resp. B-modules). A (finitely generated) left $\Lambda$-module is identified with a triple $\begin{pmatrix} X\\ Y\end{pmatrix}_{\phi}$, simply $\begin{pmatrix} X\\ Y\end{pmatrix}$ if $\phi$ is clear, where $X\in A$-mod, $Y\in B$-mod, and $\phi: N\otimes_{B}Y\longrightarrow X$ is an $A$-module homomorphism. A $\Lambda$-map $\begin{pmatrix} X\\ Y\end{pmatrix}_{\phi}\longrightarrow \begin{pmatrix} X'\\ Y'\end{pmatrix}_{\phi'}$ is identified with a pair $\begin{pmatrix}f \\ g \end{pmatrix}$, where $f\in \Hom_{A}(X, X'), g\in \Hom_{B}(Y, Y')$, such that the diagram
$$\xymatrix{
N\otimes_B Y \ar[r]^-{\phi} \ar[d]^{\text{Id}\otimes g}   &   X\ar[d]^{f}\\
N\otimes_B Y' \ar[r]^-{\phi'} &  X'}$$
commutes.
A sequence of $\Lambda$-modules
$$0\longrightarrow \begin{pmatrix} X\\ Y\end{pmatrix}_{\phi}\stackrel{\left(\begin{smallmatrix}f\\g\end{smallmatrix}\right)}{\longrightarrow}  \begin{pmatrix} X'\\ Y'\end{pmatrix}_{\phi'}\stackrel{\left(\begin{smallmatrix}f'\\g'\end{smallmatrix}\right)}{\longrightarrow} \begin{pmatrix} X''\\ Y''\end{pmatrix}_{\phi''}\longrightarrow 0$$
is exact if the sequence of $A$-modules $0\longrightarrow X \stackrel{f}\longrightarrow X' \stackrel{f'} \longrightarrow X''\longrightarrow 0$
and the sequence of $B$-modules $0\longrightarrow Y \stackrel{g}\longrightarrow Y' \stackrel{g'} \longrightarrow Y''\longrightarrow 0$ are exact.

The indecomposable projective $\Lambda$-modules are exactly $\begin{pmatrix}P\\0\end{pmatrix}$ and $\begin{pmatrix}N\otimes_B Q\\ Q\end{pmatrix}_{\text{Id}}$, where $P$ runs over indecomposable projective $A$-modules, and $Q$ runs over indecomposable projective $B$-modules.
The indecomposable injective $\Lambda$-modules are exactly $\begin{pmatrix}I\\ \Hom_{A}(N, I)\end{pmatrix}_{\Theta}$ and
$\begin{pmatrix}0\\ J\end{pmatrix}$, where $I$ runs over indecomposable injective $A$-modules, $J$ runs over indecomposable injective $B$-modules, and $\Theta: N\otimes_{B}\Hom_A(N, I)\longrightarrow I$ is an $A$-module homomorphism given by $\Theta (n\otimes f)=f(n), \forall\ n\in N, f\in\Hom_{A}(N, I)$.

\vskip10pt

The following lemma follows from \cite[p. 77]{ARS}.

\begin{lem} \label{nakayama}
Suppose that $A, B, \Lambda$ are algebras as above, $P$ an indecomposable projective $A$-module, $Q$ an indecomposable projective $B$-module. Then
$$\nu_{\Lambda}\begin{pmatrix}P\\0\end{pmatrix}\cong \begin{pmatrix}\nu_AP\\ \Hom_{A}(N, \nu_AP)\end{pmatrix}_{\Theta},\ \
\nu_{\Lambda}\begin{pmatrix}N\otimes_B Q\\ Q\end{pmatrix}_{\rm Id}\cong \begin{pmatrix}0\\ \nu_BQ\end{pmatrix},$$
where $\Theta : N\otimes_{B}\Hom_A(N, \nu_AP)\longrightarrow \nu_AP$ is an $A$-module homomorphism given by $\Theta (n\otimes f)=f(n), \forall\ n\in N, f\in\Hom_{A}(N, \nu_AP)$, and $\nu_A, \nu_B, \nu_{\Lambda}$ are the Nakayama functors of algebras $A, B, \Lambda$ respectively.
\end{lem}

\begin{lem} \label{C}\cite[Theorem 3.3]{C}
Assume that both $A$ and $B$ are Gorenstein algebras, $N={_A}N_B$ is an $A$-$B$-bimodule. Then $\begin{pmatrix}A &  N \\ 0 & B   \end{pmatrix}$ is Gorenstein if and only if ${_A}N$ and $N_B$ are finitely generated, $\proj.\dim{_A}N<\infty$ and $\proj.\dim N_B<\infty$.
\end{lem}

\begin{lem}\label{XZ}  \cite[Lemma 2.1]{XZ}
Suppose that $A, B, \Lambda$ are algebra as above, $X, X'\in A\mbox{-}{\rm mod}$ and $Y, Y'\in B\mbox{-}{\rm mod}$. Then

\vskip10pt

$(1)$ $\Ext^n_{\Lambda}(\begin{pmatrix}X'\\ 0\end{pmatrix}, \begin{pmatrix}X\\ Y\end{pmatrix}_{\phi})\cong \Ext^n_{A}(X', X),\ \forall\ n\ge 0.$

\vskip5pt

$(2)$ $\Ext^n_{\Lambda}(\begin{pmatrix}X\\ Y\end{pmatrix}_{\phi}, \begin{pmatrix}0\\ Y'\end{pmatrix})\cong \Ext^n_{B}(Y, Y'),\ \forall\ n\ge 0.$
\end{lem}

\subsection{Exceptional cycles}

Throughout this paper, $\mathcal T$ is a triangulated category {\it with Serre functor} $S$, i.e., $\mathcal{T}$ is a $k$-linear, Hom-finite, Krull-Schmidt triangulated category, with the shift functor denoted by $[1]: \mathcal{T}\longrightarrow\mathcal{T}$, and $S: \mathcal T\longrightarrow \mathcal T$ is an auto-equivalence such that for $X, Y\in\mathcal T$, there is a $k$-linear isomorphism $\Hom_{\mathcal T}(X, Y)\cong\Hom_{\mathcal T}(Y, S(X))^*$ which is functorial in $X $ and $Y$, where $(-)^*$ is the $k$-dual $\Hom_k(-, k)$. Then $S$ is a triangle functor (\cite {BK}, \cite[Appendix]{Boc}), and $\mathcal{T}$ has Auslander-Reiten triangles (\cite[p.31]{Ha}), with Auslander-Reiten translation denoted by $\tau$. Note that $S(X)\cong (\tau X)[1] \cong \tau (X[1])$ for $X\in \mathcal T$. See \cite[Theorem I.2.4]{RV}.

Let $d$ be an integer. A non-zero object $E$ of $\mathcal T$ is a {\it $d$-Calabi-Yau object}, if $S(E)\cong E[d]$. Note that a direct summand of a Calabi-Yau object is not necessarily Calabi-Yau (see \cite{CZ}).

\vskip5pt

For $X, Y\in \mathcal{T}$, let $\Hom^{\bullet}(X, Y)$ denote the complex of $k$-spaces with $\Hom^i(X, Y): = \Hom_{\mathcal{T}}(X, Y[i])$ and zero differentials. Write $\Hom^{\bullet}(X, Y)$ as $\bigoplus\limits_i \Hom_{\mathcal{T}}(X,Y[i])[-i]$. Put $\dim_k\Hom^{\bullet}(X, Y):=\sum\limits_i \dim_k\Hom_{\mathcal T}(X, Y[i])$.

\begin{defn} \ {\rm(\cite{BPP})}
{\it An exceptional $1$-cycle} in $\mathcal{T}$ is a $d$-Calabi-Yau object $E$, such that $\Hom^{\bullet}(E,E) \cong  k\oplus k[-d]$.
\end{defn}

\vskip5pt

We stress that an exceptional $1$-cycle is possibly decomposable, but this occurs if and only if $d=0$ and $\Hom_{\mathcal T}(E, E)\cong k\times k$ as algebras.

\begin{defn}\label{period}
A sequence $(E_1, E_2, \cdots, E_n)$ of objects of $\mathcal{T}$ is $S$-period with respect to $(d_1, d_2, \cdots, d_n)$, if there are integers $d_i$ such that $S(E_i)\cong E_{i+1}[d_i]$ for $1\le i\le n$, where $E_{n+1}:=E_1.$
\end{defn}

Note that the sequence $(d_1, d_2, \cdots, d_n)$ of integers in the definition is unique.

\begin{defn}{\rm(\cite{BPP})}\label{excep}
A sequence $(E_1, E_2, \cdots, E_n)$ of objects of $\mathcal{T}$ with $n\ge 2$ is an {\it exceptional $n$-cycle} with respect to a sequence of integers $(d_1, d_2, \cdots, d_n)$, if the conditions $({\rm E1})$, $({\rm E2})$ and $({\rm E3})$ are satisfied$:$

\vskip5pt

$({\rm E1})$ \ $\Hom^{\bullet}(E_i,E_i) \cong k$ for all $i;$

\vskip5pt

$({\rm E2})$ \ it is $S$-period with respect to $(d_1, d_2, \cdots, d_n)$;

\vskip 5pt

$({\rm E3})$ \ $\Hom^{\bullet}(E_i, E_j)=0$, unless $j=i$ or $j=i+1$.
$($This condition vanishes if $n = 2$.$)$
\end{defn}

In this case, $n$ is called {\it the length} of exceptional cycle $(E_1, E_2, \cdots, E_n)$. Note that each object in an exceptional $n$-cycle with $n\ge 2$ is indecomposable.

\vskip5pt

Sometimes it is easier to use the equivalent definition:

\vskip5pt

\noindent{\bf Definition \ref{excep}$'$}
A sequence $(E_1, E_2, \cdots, E_n)$ of objects of $\mathcal{T}$ with $n\ge 2$ is an {\it exceptional $n$-cycle} in $\mathcal{T}$,
if the conditions $({\rm E1'})$, $({\rm E2})$, and $({\rm E3'})$ are satisfied:

\vskip5pt

$({\rm E1'})$ \ $\Hom^{\bullet}(E_1,E_1) \cong k$;

\vskip 5pt

$({\rm E2})$ \ it is $S$-period with respect to $(d_1, d_2, \cdots, d_n)$;

\vskip 5pt

$({\rm E3'})$ \ $\Hom^{\bullet}(E_1, E_j)=0$ for $j\ge 3$.
$($This condition vanishes if $n = 2$.$)$

\vskip 5pt

\begin{thm}\cite[Lemma 1.5, Theorem 3.4]{Ha}
Let $A$ be a finite dimensional algebra. Then $A$ is Gorenstein if and only if $K^b(A\mbox{-}{\rm proj})=K^b(A\mbox{-}{\rm inj})$ in $D^b(A)$. In this case, $K^b(A\mbox{-}{\rm proj})$ has Serre functor $S$ which is induced by the Nakayama functor $\nu_A$.
\end{thm}

\section{main result}

\subsection{Perfect exceptional cycles}

\begin{defn} \label{perfect}
Let $A$ be a finite dimensional algebra, a sequence $(E_1, E_2, \cdots, E_n)$ of objects of $K^b(A\mbox{-}{\rm proj})$ with $n\ge1$ is called perfect if each $E_i$ is an $A$-module for all $1\le i\le n$.
\end{defn}

\begin{ex}
$(1)$ If $A=kQ$ with $Q$ an acyclic finite quiver, then after adjusting shift in each position, every exceptional cycle in $K^b(A\mbox{-}{\rm proj})$ can become a perfect exceptional cycle.

$(2)$ If $A = kQ/I$ is a finite dimensional gentle algebra with $A\ne k$, where $Q$ is a finite connected quiver such that the underlying graph of $Q$ is not of type $\mathbb A_3$, then after adjusting shift in each position, every exceptional cycle in $K^b(A\mbox{-}{\rm proj})$ can be a perfect exceptional cycle. Actually, according to \cite[Theorem 1.4]{GZ2}, each object in an exceptional cycle in $K^b(A\mbox{-}{\rm proj})$ is a string complex at the mouth. By \cite[Lemma 2.4]{Bob}, every string complex at the mouth is a shift of an $A$-module.

$(3)$\ If\ $A=k(1\stackrel{\alpha}\rightarrow2\stackrel{\beta}\rightarrow3)/\langle \beta\alpha\rangle$, then check that $(P(2)\rightarrow P(1), P(3)\rightarrow P(2))$ is an exceptional $2$-cycle in $K^b(A\mbox{-}{\rm proj})$, but no matter how to adjust shift in each position, it would not become a perfect exceptional cycle in $K^b(A\mbox{-}{\rm proj})$.
\end{ex}

\begin{lem} \label{iff1}
Let $A$ be a Gorenstein algebra, a $d$-Calabi-Yau object $E$ in $K^b(A\mbox{-}{\rm proj})$ is a perfect exceptional $1$-cycle if and only if $E$ is an $A$-module and
$${\rm when}\ d\ne 0,\ \Ext^t_A(E, E)\cong
\begin{cases}
 k, &  {\rm if}\ t=0;\\
 k, &  {\rm if}\ t=d;  \\
 0, &  {\rm otherwise.}
\end{cases}\ \ \ \ \ \ (i)$$
and $${\rm when}\  d=0,\ \Ext^t_A(E, E)\cong
\begin{cases}
k^2, &  {\rm if}\ t=0;  \\
0,  &  {\rm otherwise.}
\end{cases}\ \ \ \ \ \ (ii)$$
\end{lem}

\begin{proof}
Assume that $E$ is an $A$-module and satisfies $(i)$ and $(ii)$, then
\begin{align*}
\Hom^{\bullet}(E, E)& =\bigoplus\limits_t \Hom_{K^b(A\mbox{-}{\rm proj})}(E,E[t])[-t]\\
&\cong\bigoplus\limits_t \Hom_{D^b(A)}(E,E[t])[-t]\\
&\cong \bigoplus\limits_t \Ext^t_A(E, E)[-t]\\
&\cong k\oplus k[-d].
\end{align*}
Notice that the last equation holds whether $d=0$ or not, thus $E$ is a perfect exceptional $1$-cycle.

On the other hand, if $E$ is a perfect exceptional $1$-cycle, then $E$ is an $A$-module and $\Hom^{\bullet}(E, E)=k\oplus k[-d]$.

When $d=0$, it follows that $\Hom^{\bullet}(E, E)\cong k^2\cong \bigoplus\limits_t \Ext^t_A(E, E)[-t]$. Thus one has
$$\Ext^t_A(E, E)\cong
\begin{cases}
k^2, &  {\rm if}\ t=0;  \\
0, &  {\rm otherwise.}
\end{cases}$$

When $d\ne0$, it follows that $\Hom^{\bullet}(E, E)\cong k\oplus k[-d]\cong \bigoplus\limits_t \Ext^t_A(E, E)[-t]$. Hence
$$\Ext^t_A(E, E)\cong
\begin{cases}
k,  &  {\rm if}\ t=0;\\
k,  &  {\rm if}\ t=d;  \\
0, &  {\rm otherwise.}
\end{cases}$$
\end{proof}

\begin{lem}  \label{iff2}
Suppose that $A$ is a Gorenstein algebra, the sequence $(E_1, E_2, \cdots, E_n)$ in $K^b(A\mbox{-}{\rm proj})$ is $S$-period with respect to $(d_1, d_2, \cdots, d_n)$.
Then the following are equivalent:

$(i)$\  $(E_1, E_2, \cdots, E_n)$ is a perfect exceptional $n$-cycle;

$(ii)$\  Each $E_i$ is an $A$-module for $1\le i\le n$, and
$$\Ext^t_A(E_i, E_j)\cong
\begin{cases}
k, & {\rm if}\ j=i,\ t=0; \\
k,  & {\rm if}\  j=i+1,\ t=d_i;\\
0, & {\rm otherwise.}
\end{cases}$$

$(iii)$\  Each $E_i$ is an $A$-module for $1\le i\le n$, and
$$\Ext^t_A(E_1, E_j)\cong
\begin{cases}
k, & {\rm if}\ j=1,\ t=0; \\
k,  & {\rm if}\  j=2,\ t=d_1;\\
0, & {\rm otherwise.}
\end{cases}$$
\end{lem}

\begin{proof}
$(i)\Rightarrow(ii)$ According to the definition of a perfect exceptional cycle, then each $E_i$ is an $A$-module for all $1\le i\le n$, and $\Hom^{\bullet}(E_i, E_i)=k$ for all $1\le i\le n$, and $\Hom^{\bullet}(E_i, E_j)=0$, unless $j=i$ or $j=i+1$. Then
\begin{align*}
k\cong\Hom^{\bullet}(E_i, E_i)& =\bigoplus\limits_t \Hom_{K^b(A\mbox{-}{\rm proj})}(E_i,E_i[t])[-t]\\
&\cong\bigoplus\limits_t \Hom_{D^b(A)}(E_i,E_i[t])[-t]\\
&\cong \bigoplus\limits_t \Ext^t_A(E_i, E_i)[-t].
\end{align*}
It follows that $$\Ext^t_A(E_i, E_i)\cong
\begin{cases}
k, & {\rm if}\ t=0; \\
0,  & {\rm otherwise}.
\end{cases}$$
Thus
\begin{align*}
\Ext^t_A(E_i, E_{i+1})& \cong\Hom_{K^b(A\mbox{-}{\rm proj})}(E_i,E_{i+1}[t])\\
&\cong\Hom_{K^b(A\mbox{-}{\rm proj})}(E_i,S(E_i)[t-d_i])\\
&\cong\Hom_{K^b(A\mbox{-}{\rm proj})}(E_i,S(E_i[t-d_i]))\\
&\cong D\Hom_{K^b(A\mbox{-}{\rm proj})}(E_i[t-d_i],E_i)\\
&\cong D\Hom_{K^b(A\mbox{-}{\rm proj})}(E_i,E_i[d_i-t])\\
&\cong D\Ext^{d_i-t}_A(E_i, E_i)\\
&\cong
\begin{cases}
k, & {\rm if}\ t=d_i; \\
0,  & {\rm otherwise}.\end{cases}
\end{align*}

If $j\ne i$ and $j\ne i+1$, then $0=\Hom^{\bullet}(E_i, E_j)=\bigoplus\limits_t \Ext^t_A(E_i, E_j)[-t]$.
All together, we have
$$\Ext^t_A(E_i, E_j)\cong
\begin{cases}
k, & {\rm if}\ j=i,\ t=0; \\
k,  & {\rm if}\  j=i+1,\ t=d_i;\\
0, & {\rm otherwise.}\end{cases}$$

$(ii)\Rightarrow(iii)$\ It is clear.

$(iii)\Rightarrow(i)$\ If each $E_i$ is an $A$-module for $i\ge1$, then
\begin{align*}
\Hom^{\bullet}(E_1, E_j)& \cong \bigoplus\limits_t \Ext^t_A(E_1, E_j)[-t]\\
&\cong
\begin{cases}
k, & {\rm if}\ j=1; \\
k[-d_1],  & {\rm if}\  j=2;\\
0, & {\rm otherwise.}
\end{cases}
\end{align*}
By Definition \ref{excep}$'$, $(E_1, E_2, \cdots, E_n)$ is a perfect exceptional $n$-cycle.
\end{proof}

\subsection{Serre functor for triangular matrix algebras}

Throughout this section, assume that $A$ and $B$ are two finite dimensional Gorenstein algebras, $\Lambda=\begin{pmatrix}A & N\\ 0 & B \end{pmatrix}$, where $N$ is an $A$-$B$-bimodule. The Serre functor in $K^b(A\mbox{-}{\rm proj})$ (resp. $K^b(B\mbox{-}{\rm proj}), K^b(\Lambda\mbox{-}{\rm proj})$) is denoted by $S_A$ (resp. $S_B, S_{\Lambda}$), and the Nakayama functor is denote by $\nu_A$ (resp. $\nu_B, \nu_{\Lambda}$).

\begin{lem}
Let $A$ and $B$ be Gorenstein algebras, $(E_1, E_2, \cdots, E_n)$  a perfect exceptional $n$-cycle in $K^b(A\mbox{-}{\rm proj})$,
$(F_1, F_2, \cdots, F_m)$ a perfect exceptional $m$-cycle in $K^b(B\mbox{-}{\rm proj})$.
Then $\Lambda=\begin{pmatrix}A & E_n\otimes_k D(F_1)\\ 0 & B \end{pmatrix}$ is Gorenstein.
Consequently, $K^b(\Lambda\mbox{-}{\rm proj})$ has Serre functor.
\end{lem}

\begin{proof}
Since $E_n\in K^b(A\mbox{-}\proj)$, we have $\proj.\dim {_A}(E_n)<\infty$. It follows from $F_1\in  K^b(A\mbox{-}\proj)= K^b(A\mbox{-}{\rm inj})$ that
$\inj.\dim {_B}(F_1)<\infty$. Hence $\proj.\dim\ (D(F_1))_B<\infty$. Then one has $$\proj.\dim {_A}(E_n\otimes_k D(F_1))=\proj.\dim {_A}(E_n)<\infty$$ and
$$\proj.\dim (E_n\otimes_k D(F_1))_B=\proj.\dim\ (D(F_1))_B<\infty.$$
By Lemma \ref{C}, $\Lambda$ is Gorenstein.
\end{proof}

\begin{lem} \label{interpret}
Let $A$ be a Gorenstein algebra, $X, Y\in A\mbox{-}{\rm mod}$ such that $S_A(X)\cong Y[d]$ in $K^b(A\mbox{-}{\rm proj})$, where
$S_A$ is the Serre functor in $K^b(A\mbox{-}{\rm proj})$. Then ${\rm proj.}\dim X=d={\rm inj.}\dim Y$.
\end{lem}

\begin{proof}
Assume that ${\rm proj.}\dim X=n$. Take a minimal projective resolution of $X$
$$P_{\bullet}:\ \ \ \ \ \ \ \ 0\longrightarrow P_n \stackrel{\alpha_n}{\longrightarrow} P_{n-1}\longrightarrow \cdots \longrightarrow P_1\stackrel{\alpha_1}{\longrightarrow} P_0 \longrightarrow 0,$$
where $\alpha_n$ is not a splitting monomorphism. Applying the Nakayama functor $\nu_A$ to $P_{\bullet}$, one has
$$\nu_A P_{\bullet}:\ \ \ \ \ \ \ \ 0\longrightarrow \nu_A P_n \stackrel{\nu_A \alpha_n}{\longrightarrow} \nu_A P_{n-1}\longrightarrow \cdots \longrightarrow \nu_A P_1\stackrel{\nu_A \alpha_1}{\longrightarrow} \nu_A P_0 \longrightarrow 0.$$

Since $S_A(X)\cong S_A(P_{\bullet})\cong \nu_A P_{\bullet}\cong Y[d]$ in $K^b(A\mbox{-}{\rm proj})$ and $Y[d]$ is a stalk complex, hence
$${\rm H}_i(\nu_A P_{\bullet})\cong {\rm H}_i(Y[d])\cong
\begin{cases}
Y, & {\rm if}\ i=-d;\\
0, &  {\rm otherwise.}
\end{cases}$$
We claim that $d=n$. Otherwise, $H_n(\nu_A P_{\bullet})=0$. It follows that $\nu \alpha_n$ is a monomorphism. Since $\nu_A P_n$ and $\nu_A P_{n-1}$ are injective modules, thus $\nu_A \alpha_n$ splits. So $\alpha_n=\nu^{-1}_A\nu_A \alpha_n$ splits. It's a contradiction.
Then $\nu_A P_{\bullet}$ is an injective resolution of $Y$. We claim that $\nu_A \alpha_1$ is not a splitting epimorphism (otherwise, $\alpha_1$ is an splitting epimorphism, it's a contradiction). Hence, $\nu_A P_{\bullet}$ is a minimal injective resolution of $Y$, it follows that ${\rm inj.}\dim Y=d=n={\rm proj.}\dim X$.
\end{proof}

\begin{rem}
If $(E_1, E_2, \cdots, E_n)$ is a perfect exceptional $n$-cycle with respect to $(c_1, c_2, \cdots, c_n)$ in $K^b(A\mbox{-}{\rm proj})$, then by above lemma, the sequence of integers $(c_1, c_2, \cdots, c_n)$ has an interpretation of representation theory: each integer $c_i$ equals to the projective dimension of $E_i$ $($or the injective dimension of $E_{i+1})$, hence non-negative.
\end{rem}

\begin{lem}\label{S1}
Assume that $A, B, \Lambda$ are as above, $X, Y\in A\mbox{-}{\rm mod}$ such that $S_A(X)\cong Y[c]$ in $K^b(A\mbox{-}{\rm proj})$, $N$ is an $A$-$B$-bimodule satisfing $\Ext^t_A(N, Y)=0$ for $t\ge 1$. Then
$$S_{\Lambda}\begin{pmatrix}X\\0\end{pmatrix}\cong \begin{pmatrix}Y\\ \Hom_A(N,Y)\end{pmatrix}_{\Theta}[c],$$
where $\Theta : N\otimes_{B}\Hom_A(N, Y)\longrightarrow Y$ is given by $\Theta (n\otimes f)=f(n), \forall\ n\in N, f\in\Hom_{A}(N, Y).$
\end{lem}

\begin{proof}
By Lemma \ref{interpret}, ${\rm proj.}\dim X=c$, take a (minimal) projective resolution of $X$,
$$0 \longrightarrow P_c\longrightarrow \cdots \longrightarrow P_1\longrightarrow P_0 \stackrel{\alpha}{\longrightarrow} X \longrightarrow0.$$

It follows from the proof of Lemma \ref{interpret} that
$$0 \longrightarrow Y\stackrel{\beta}{\longrightarrow} \nu_A P_c\longrightarrow\cdots \longrightarrow \nu_A P_1\longrightarrow \nu_A P_0\longrightarrow 0$$ is a (minimal) injective resolution of $Y$. Applying $\Hom_A(N,-)$ to above injective resolution of $Y$, one gets
$$0 \longrightarrow \Hom_A(N, \nu_A P_c)\longrightarrow \cdots \longrightarrow \Hom_A(N, \nu_A P_1)\longrightarrow \Hom_A(N, \nu_A P_0)\longrightarrow 0.$$
Actually, the $t$-th cohomology group is $\Ext^t_A(N, Y)$.
Since $\Ext^t_A(N, Y)=0$ for $t \ge 1$, then
$$0 \longrightarrow \Hom_A(N, Y)\stackrel{\beta\circ(-)}{\longrightarrow} \Hom_A(N, \nu_A P_c)\longrightarrow \cdots \longrightarrow \Hom_A(N, \nu_A P_0)\longrightarrow 0$$ is exact. Then we have exact sequences
$$0\longrightarrow \begin{pmatrix}P_c\\ 0 \end{pmatrix}\longrightarrow \cdots  \longrightarrow \begin{pmatrix}P_1\\ 0 \end{pmatrix}\longrightarrow \begin{pmatrix}P_0\\ 0 \end{pmatrix}\stackrel{\left(\begin{smallmatrix}\alpha\\0\end{smallmatrix}\right)}{\longrightarrow} \begin{pmatrix}X\\ 0 \end{pmatrix}\longrightarrow 0$$
and
$$0\longrightarrow \begin{pmatrix}Y\\ \Hom_A(N, Y) \end{pmatrix}_{\Theta}\stackrel{\left(\begin{smallmatrix}\beta\\ \beta\circ(-)\end{smallmatrix}\right)}{\longrightarrow} \begin{pmatrix} \nu_A P_c \\  \Hom_A(N ,\nu_A P_c)\end{pmatrix}_{\Theta}\longrightarrow \cdots  \longrightarrow  \begin{pmatrix}\nu_A P_0\\ \Hom_A(N ,\nu_A P_0) \end{pmatrix}_{\Theta}\longrightarrow0,$$
where $\Theta: N\otimes_{B}\Hom_A(N, Z)\longrightarrow Z$ is given
by $\Theta (n\otimes f)=f(n), \forall\ n\in N, f\in\Hom_{A}(N, Z), Z\in \{Y, \nu_A P_c,\cdots, \nu_A P_0\}.$

That is, $P^\bullet=\Bigg( 0\longrightarrow \begin{pmatrix}P_c\\ 0 \end{pmatrix}\longrightarrow \cdots  \longrightarrow  \begin{pmatrix}P_0\\ 0 \end{pmatrix}\longrightarrow 0\Bigg)$
is a projective resolution of $\begin{pmatrix}X\\ 0 \end{pmatrix}$, and
$I^\bullet=\Bigg(0\longrightarrow \begin{pmatrix} \nu_A P_c \\  \Hom_A(N ,\nu_A P_c)\end{pmatrix}_{\Theta}\longrightarrow \cdots  \longrightarrow  \begin{pmatrix}\nu_A P_0\\ \Hom_A(N ,\nu_A P_0) \end{pmatrix}_{\Theta}\longrightarrow0\Bigg)$
is an injective resolution of $\begin{pmatrix}Y\\ \Hom_A(N, Y)\end{pmatrix}_{\Theta}$.

By Lemma \ref{nakayama}, $\nu_{\Lambda}P^\bullet\cong I^\bullet$, so $S_{\Lambda}\begin{pmatrix}X  \\  0 \end{pmatrix}\cong S_{\Lambda}(P^\bullet)\cong \nu_{\Lambda}P^{\bullet}\cong  I^\bullet\cong \begin{pmatrix}Y \\  \Hom_A(N, Y)\end{pmatrix}_{\Theta}[c]$.
\end{proof}

The following Lemma can be proved by the same argument as Lemma \ref{S1}, we omit the proof.

\begin{lem}\label{S2}
Assume that $A, B, \Lambda$ are as above, $V, W\in B\mbox{-}{\rm mod}$ such that $S_B(V)\cong W[d]$ in $K^b(B\mbox{-}{\rm proj})$, $N$ is an $A$-$B$-bimodule satisfing $\Tor_t^B(N, V)=0$ for $t\ge 1$. Then
$$S_{\Lambda}\begin{pmatrix}N\otimes_B V\\V\end{pmatrix}_{\rm Id}\cong \begin{pmatrix}0\\ W\end{pmatrix}[d].$$
\end{lem}

\begin{lem}\label{S3}
Assume that $A, B, \Lambda$ are as above, $X, Y\in A\mbox{-}{\rm mod}$ such that $S_A(X)\cong Y[c]$ in $K^b(A\mbox{-}{\rm proj})$, $V, W\in B\mbox{-}{\rm mod}$ such that $S_B(V)\cong W[d]$ in $K^b(B\mbox{-}{\rm proj})$, $N$ is an $A$-$B$-bimodule satisfing $\Ext^t_A(N, Y)=\begin{cases}W, & {\rm if}\ t=c;\\ 0, &{\rm otherwise}\end{cases}$ and $\Tor_t^B(N, V)=\begin{cases}X, & {\rm if}\ t=d;\\ 0, &{\rm otherwise}\end{cases}$. Then
$$S_{\Lambda}\begin{pmatrix}0 \\V\end{pmatrix}\cong \begin{pmatrix}Y\\ 0\end{pmatrix}[c+d+1].$$
\end{lem}

\begin{proof}
Take a projective resolution of $X$,
$$0 \longrightarrow P_c\longrightarrow \cdots \longrightarrow P_1\longrightarrow P_0 \stackrel{\alpha_1}{\longrightarrow} X\longrightarrow 0.$$

Since $S_A(X)\cong Y[c]$, then
$$0 \longrightarrow Y\stackrel{\gamma_1}{\longrightarrow} \nu_A P_c\longrightarrow\cdots \longrightarrow \nu_A P_1\longrightarrow \nu_A P_0\longrightarrow 0$$ is an injective resolution of $Y$. Applying $\Hom_A(N,-)$ to above injective resolution of $Y$, one gets
$$0 \longrightarrow \Hom_A(N, \nu_A P_c)\longrightarrow \cdots \longrightarrow \Hom_A(N, \nu_A P_1)\longrightarrow \Hom_A(N, \nu_A P_0)\longrightarrow 0.$$
The $t$-th cohomology group is $\Ext^t_A(N, Y)$. Since $\Ext^t_A(N, Y)=\begin{cases}W, & {\rm if}\ t=c;\\ 0, &{\rm otherwise}\end{cases}$, then we get an exact sequence
$$0 \longrightarrow \Hom_A(N, \nu_A P_c)\longrightarrow \cdots \longrightarrow \Hom_A(N,\nu_AP_0)\stackrel{\alpha_2}{\longrightarrow} W \longrightarrow 0.$$

\vskip10pt

Take a projective resolution of $V$,
$$0 \longrightarrow Q_d\longrightarrow \cdots \longrightarrow Q_1\longrightarrow Q_0 \stackrel{\gamma_2}{\longrightarrow} V\longrightarrow0$$
Since $S_B(V)\cong W[d]$, then
$$0 \longrightarrow W \stackrel{\beta_2}{\longrightarrow} \nu_B Q_d\longrightarrow\cdots \longrightarrow \nu_B Q_1 \longrightarrow \nu_B Q_0\longrightarrow 0$$
is an injective resolution of $W$. Applying $N\otimes_B- $ to above projective resolution of $V$, we get
$$0 \longrightarrow N\otimes_B Q_d\longrightarrow \cdots \longrightarrow N\otimes_B Q_1\longrightarrow N\otimes_B Q_0\longrightarrow 0.$$
The $t$-th homology group is $\Tor_t^B(N, V)$. Since $\Tor_t^B(N, V)=\begin{cases}X, & {\rm if}\ t=d;\\ 0, &{\rm otherwise}\end{cases}$, there is an exact sequence
$$0 \longrightarrow X\stackrel{\beta_1}{\longrightarrow} N\otimes_B P_d\longrightarrow \cdots \longrightarrow N\otimes_B P_0\longrightarrow 0.$$

Then we have exact sequences
$$0\rightarrow \begin{pmatrix}P_c\\0 \end{pmatrix} \rightarrow \cdots \rightarrow \begin{pmatrix}P_{0}\\0 \end{pmatrix}   \stackrel{\left(\begin{smallmatrix}\beta_1\alpha_1\\ 0\end{smallmatrix}\right)}{\longrightarrow} \begin{pmatrix} N\otimes_B  Q_d\\ Q_d \end{pmatrix}_{\rm Id}\rightarrow \cdots  \rightarrow  \begin{pmatrix}N\otimes_B Q_0 \\Q_0 \end{pmatrix}_{\rm Id}\stackrel{\left(\begin{smallmatrix} 0\\  \gamma_2 \end{smallmatrix}\right)}{\longrightarrow}\begin{pmatrix} 0\\ V \end{pmatrix}\rightarrow 0$$
and
$$0\longrightarrow \begin{pmatrix}Y\\ 0 \end{pmatrix}\stackrel{\left(\begin{smallmatrix}\gamma_1\\ 0\end{smallmatrix}\right)}{\longrightarrow} \begin{pmatrix}\nu_A P_c  \\ \Hom_A(N, \nu_A P_c)\end{pmatrix}_{\Theta}\longrightarrow \cdots\longrightarrow \begin{pmatrix}\nu_A P_0  \\ \Hom_A(N, \nu_A P_0)\end{pmatrix}_{\Theta}\stackrel{\left(\begin{smallmatrix}0\\ \beta_2\alpha_2\end{smallmatrix}\right)}{\longrightarrow}\begin{pmatrix} 0 \\  \nu_B Q_d\end{pmatrix} \hskip80pt$$
$$\hskip270pt\longrightarrow \cdots  \longrightarrow \begin{pmatrix}  0\\ \nu_B Q_0 \end{pmatrix}\longrightarrow0$$
It follows that $S_{\Lambda}\begin{pmatrix}0\\V\end{pmatrix}\cong \begin{pmatrix}Y\\ 0\end{pmatrix}[c+d+1].$

\end{proof}

\subsection{Product of two perfect exceptional cycles}

\begin{defn}\label{prod}
Suppose that $A$ and $B$ are finite dimensional algebras, $E_*=(E_1, E_2, \cdots, E_n)$ is a sequence of $A$-modules with $n\ge 1$,
$F_*=(F_1, F_2, \cdots, F_m)$ a sequence of $B$-modules with $m\ge1$. Let $\Lambda=\begin{pmatrix}A & N\\ 0 & B \end{pmatrix}$ with $A$-$B$-bimodule $N=E_n\otimes_k D(F_1)$, $\Psi: E_n\otimes_k D(F_1)\otimes_B F_1\longrightarrow E_n$ be an $A$-module homomorphism given by $\Psi(a\otimes_k f\otimes_B b)=f(b)a, \forall\ a\in E_n, f\in D(F_1), b\in F_1$.
Then the sequence of $\Lambda$-modules
$$(\begin{pmatrix}E_1\\0\end{pmatrix}, \cdots, \begin{pmatrix}E_{n-1}\\0\end{pmatrix}, \begin{pmatrix}E_n\\F_1\end{pmatrix}_{\Psi}, \begin{pmatrix}0\\ F_2\end{pmatrix}, \cdots, \begin{pmatrix}0\\ F_m\end{pmatrix})$$
is called the product of $E_*$ and $F_*$, denote it by $E_*\mathop{\boxtimes} F_*$.
\end{defn}

\begin{lem} \label{key}
Let $A$ and $B$ be finite dimensional algebras, $E, X\in A\mbox{-}{\rm mod}$ and $F, Y\in B\mbox{-}{\rm mod}$.\\
$(1)$ There exists an isomorphism of $B$-module
$$\Ext^n_A(E\otimes_k D(F), X)\cong F\otimes_k\Ext^n_A(E, X),\ \forall\ n\ge0.$$
$(2)$ There exists an isomorphism of $A$-module
$$\Tor_n^B(E\otimes_kD(F), Y)\cong E\otimes_kD\Ext^n_B(Y, F),\ \forall\ n\ge0.$$
\end{lem}

\begin{proof}
$(1)$  Take an injective resolution $0\longrightarrow X\longrightarrow I^\bullet$. Since an exact functor commutes with the cohomology functor, there is an isomorphism of $k$-linear space for $n\ge0$
\begin{align*}
\Ext^n_A(E\otimes_k D(F), X)&\cong {\rm H}^n(\Hom_A(E\otimes_k D(F), I^\bullet)) \\
&\cong {\rm H}^n(\Hom_k(D(F), \Hom_A(E, I^\bullet))\\
&\cong \Hom_k(D(F),{\rm H}^n(\Hom_A(E, I^\bullet)))\\
&\cong \Hom_k(D(F), \Ext^n_A(E, X))\\
&\cong F\otimes_k \Ext^n_A(E, X).
\end{align*}
On the other hand, $\Ext^n_A(E\otimes_k D(F), X)$ has the left $B$-module structure given by the right $B$-module structure of $E\otimes_kD(F)$, hence for $n\ge0$, there exists a $B$-module isomorphism
$$\Ext^n_A(E\otimes_k D(F), X)\cong F\otimes_k\Ext^n_A(E, X).$$
$(2)$ Take a projective resolution $P^\bullet\longrightarrow Y\longrightarrow 0$. Since an exact functor commutes with the homology functor, there is an isomorphism of $k$-linear space for $n\ge0$
\begin{align*}
\Tor^B_n(E\otimes_k D(F), Y)& \cong {\rm H}_n(E\otimes_k D(F)\otimes_B P^\bullet)\\
& \cong {\rm H}_n(E\otimes_k D\Hom_B(P^\bullet, F))\\
& \cong E\otimes_k{\rm H}_n(D\Hom_B(P^\bullet, F))\\
& \cong E\otimes_kD{\rm H}^n(\Hom_B(P^\bullet, F))\\
&\cong E\otimes_kD\Ext_B^n(Y,F).
\end{align*}
Similarly, $\Tor^B_n(E\otimes_k D(F), Y)$ has the left $A$-module structure given by the left $A$-module structure of $E\otimes_kD(F)$, hence for $n\ge0$, there exists an $A$-module isomorphism
$$\Tor^B_n(E\otimes_k D(F), Y)\cong F\otimes_kD\Ext^n_A(Y, F).$$
\end{proof}

\begin{prop}\label{S4}
Let $A$ and $B$ be Gorenstein algebras, $E_*=(E_1, E_2, \cdots, E_n)$ a perfect exceptional $n$-cycle with respect to $(c_1, c_2, \cdots, c_n)$ in $K^b(A\mbox{-}{\rm proj})$, $F_*=(F_1, F_2, \cdots, F_m)$ a perfect exceptional $m$-cycle with respect to $(d_1, d_2, \cdots, d_m)$ in $K^b(B\mbox{-}{\rm proj})$. Let $\Lambda=\begin{pmatrix}A & N\\ 0 & B \end{pmatrix}$ with $A$-$B$-bimodule $N=E_n\otimes_k D(F_1)$. Then $E_*\mathop{\boxtimes} F_*$ is $S_{\Lambda}$-period with respect to $(c_1,\cdots,c_{n-1}, d_1, \cdots, d_{m-1}, c_n+d_m+1)$ in $K^b(\Lambda\mbox{-}{\rm proj})$.
\end{prop}

\begin{proof}
We only give a proof of the case $n\ge 2, m\ge 2$, since the other three cases: $\text{(i)}\ n=1, m\ge 2;$ $\text{(ii)}\ n\ge 2, m=1; \ \text{(iii)}\ n=1, m=1$ are proved identically.

According to Lemma \ref{key} and Lemma \ref{iff2},
$$\Ext_A^t(N, E_i)\cong F_1\otimes_k\Ext_A^t(E_n, E_i)=\begin{cases}
F_1, & {\rm if}\ i=1,\ t=c_n;\\
F_1, & {\rm if}\ i=n,\ t=0;\ \ \ \ \ \ \ \ (*1)\\
0, & {\rm otherwise.}
\end{cases}$$

$$\Tor^B_t(N, F_i)\cong E_n\otimes D\Ext_B^t(F_i, F_1)=\begin{cases}
E_n,  & {\rm if}\ i=1,\ t=0 ;\\
E_n,  & {\rm if}\ i=m,\ t=d_m ;\ \ \ \ (*2)\\
0, & {\rm otherwise.}
\end{cases}$$

Since $S_A(E_i)\cong E_{i+1}$ for $1\le i \le n-1$, it follows from Lemma \ref{S1} and $(*1)$ that
$$S_{\Lambda}\begin{pmatrix}E_i  \\  0 \end{pmatrix}\cong \begin{pmatrix}E_{i+1} \\  0\end{pmatrix}[c_i],
\ 1\le i\le n-2\ \text{and}\ S_{\Lambda}\begin{pmatrix}E_{n-1}  \\ 0 \end{pmatrix}\cong \begin{pmatrix}E_n \\  F_1\end{pmatrix}_{\Psi}[c_{n-1}].$$

Similarly, due to $S_B(F_i)\cong F_{i+1}$ for $1\le i \le m-1$, according to Lemma \ref{S2} and $(*2)$,

$$S_{\Lambda}\begin{pmatrix}E_n  \\  F_1 \end{pmatrix}_{\Psi}\cong \begin{pmatrix}0 \\  F_2\end{pmatrix}[d_1]\ \text{and}\ S_{\Lambda}\begin{pmatrix}0  \\  F_i \end{pmatrix}\cong \begin{pmatrix}0 \\  F_{i+1}\end{pmatrix}[d_i],\ 2\le i\le m-1.$$

Because $S_A(E_n)\cong E_1[c_n]$ and $S_B(F_m)\cong F_1[d_m]$, by Lemma \ref{S3}, $(*1)$ and $(*2)$,

$$S_{\Lambda}\begin{pmatrix}0  \\  F_m \end{pmatrix}\cong \begin{pmatrix}E_1 \\  0\end{pmatrix}[c_n+d_m+1].$$

\end{proof}

\begin{lem}\label{1x1}
Let $A$ and $B$ be Gorenstein algebras, $E$ a perfect exceptional $1$-cycle with $S_A(E)\cong E[c]$ in $K^b(A\mbox{-}{\rm proj})$,
$F$ a perfect exceptional $1$-cycle with $S_B(F)\cong F[d]$ in $K^b(B\mbox{-}{\rm proj})$. Assume that $\Lambda=\begin{pmatrix}A & N\\ 0 & B \end{pmatrix}$ with $A$-$B$-bimodule $N=E\otimes_k D(F)$. Then $E\mathop{\boxtimes} F=\begin{pmatrix}E\\F\end{pmatrix}_{\Psi}$ is an exceptional $1$-cycle with
$S_\Lambda\begin{pmatrix}E\\ F \end{pmatrix}_{\Psi}\cong \begin{pmatrix}E\\ F \end{pmatrix}_{\Psi}[c+d+1]$.
\end{lem}

\begin{proof}
By Proposition \ref{S4}, one has $S_\Lambda\begin{pmatrix}E\\ F \end{pmatrix}_{\Psi}\cong \begin{pmatrix}E\\ F \end{pmatrix}_{\Psi}[c+d+1]$.
Due to Lemma \ref{interpret}, $c\ge 0, d\ge 0$, hence $c+d+1>0$.  According to Lemma \ref{iff1}, it suffices to prove

$$\Ext_{\Lambda}^t(\begin{pmatrix}E\\ F\end{pmatrix}_{\Psi}, \begin{pmatrix}E\\ F\end{pmatrix}_{\Psi})=\begin{cases}k  &  \text{if}\ t=0; \\ k & \text{if}\ t=c+d+1;\\
0 &\text{otherwise.} \end{cases}$$

By the exact sequence of $\Lambda$-modules
$$0\longrightarrow \begin{pmatrix}E\\ 0\end{pmatrix}\longrightarrow \begin{pmatrix}E\\ F\end{pmatrix}_{\Psi}\longrightarrow\begin{pmatrix}0\\ F\end{pmatrix}\longrightarrow 0,$$
there is a long exact sequence of $k$-linear spaces

$$\cdots\rightarrow\Ext_{\Lambda}^t(\begin{pmatrix}E\\ F\end{pmatrix}_{\Psi}, \begin{pmatrix}E\\ 0\end{pmatrix})
\rightarrow\Ext_{\Lambda}^t(\begin{pmatrix}E\\ F\end{pmatrix}_{\Psi}, \begin{pmatrix}E\\ F\end{pmatrix}_{\Psi})
\rightarrow\Ext_{\Lambda}^t(\begin{pmatrix}E\\ F\end{pmatrix}_{\Psi}, \begin{pmatrix}0\\ F\end{pmatrix})\rightarrow\cdots$$

Consider \begin{align*}\Ext_{\Lambda}^t(\begin{pmatrix}E\\ F\end{pmatrix}_{\Psi}, \begin{pmatrix}E\\ 0\end{pmatrix})&
\cong \Hom_{K^b(\Lambda\mbox{-}{\rm proj})}(\begin{pmatrix}E\\ F\end{pmatrix}_{\Psi}, \begin{pmatrix}E\\ 0\end{pmatrix}[t])\\
&\cong \Hom_{K^b(\Lambda\mbox{-}{\rm proj})}(\begin{pmatrix}E\\ 0\end{pmatrix}[t],S_{\Lambda}\begin{pmatrix}E\\ F\end{pmatrix}_{\Psi})\\
&\cong \Hom_{K^b(\Lambda\mbox{-}{\rm proj})}(\begin{pmatrix}E\\ 0\end{pmatrix}[t],\begin{pmatrix}E\\ F\end{pmatrix}_{\Psi}[c+d+1])\\
&\cong \Hom_{K^b(\Lambda\mbox{-}{\rm proj})}(\begin{pmatrix}E\\ 0\end{pmatrix},\begin{pmatrix}E\\ F\end{pmatrix}_{\Psi}[c+d+1-t])\\
&\cong\Ext_{\Lambda}^{c+d+1-t}(\begin{pmatrix}E\\ 0\end{pmatrix}_{\Psi}, \begin{pmatrix}E\\ F\end{pmatrix}_{\Psi})\ \ \ \ \ (\text{By\ Lemma}\ \ref{XZ})\\
&\cong\Ext_A^{c+d+1-t}(E, E).
\end{align*}
It follows from Lemma \ref{XZ} that $\Ext_{\Lambda}^t(\begin{pmatrix}E\\ F\end{pmatrix}_{\Psi}, \begin{pmatrix}0\\ F\end{pmatrix})\cong \Ext_B^t(F, F)$.

Hence above long exact sequence becomes

$$\cdots\rightarrow\Ext_A^{c+d+1-t}(E, E) \rightarrow\Ext_{\Lambda}^t(\begin{pmatrix}E\\ F\end{pmatrix}_{\Psi}, \begin{pmatrix}E\\ F\end{pmatrix}_{\Psi})\rightarrow\Ext_B^t(F, F)\rightarrow\ \Ext_A^{c+d-t}(E,E)\rightarrow \cdots$$

Assume that $c\ne 0$ and $d\ne 0$ (the other three cases: $\text{(i)}\ c=0, d\ne 0;$ $\text{(ii)}\ c\ne 0, d=0;$ $\text{(iii)}\ c=0, d=0$ are proved similarly).

If $c+d+1-t\ne 0, c+d+1-t\ne c$ and $t\ne 0, t\ne d$, then $\Ext_A^{c+d+1-t}(E, E)=0$ and $\Ext_B^t(F, F)=0$, it follows that $\Ext_{\Lambda}^t(\begin{pmatrix}E\\ F\end{pmatrix}_{\Psi}, \begin{pmatrix}E\\ F\end{pmatrix}_{\Psi})=0$.

If $c+d+1-t=0$, then $\Ext_A^{c+d+1-t}(E, E)=k$ and $t\ne 0,\ t\ne d$. Hence one has $\Ext_B^t(F, F)=0$ and $\Ext_B^{t-1}(F, F)=0$. It follows that $\Ext_{\Lambda}^t(\begin{pmatrix}E\\ F\end{pmatrix}_{\Psi}, \begin{pmatrix}E\\ F\end{pmatrix}_{\Psi})=k$.

If $c+d+1-t=c$, then $\Ext_A^{c+d+1-t}(E, E)=k,\ \Ext_B^{t-1}(F, F)=k$, so
$\Ext_{\Lambda}^t(\begin{pmatrix}E\\ F\end{pmatrix}_{\Psi}, \begin{pmatrix}E\\ F\end{pmatrix}_{\Psi})=0$.

If $t=0$, then $\Ext_B^t(F, F)=k, \Ext_A^{c+d+1-t}(E, E)=0$, hence $\Ext_{\Lambda}^t(\begin{pmatrix}E\\ F\end{pmatrix}_{\Psi}, \begin{pmatrix}E\\ F\end{pmatrix}_{\Psi})=k$.

If $t=d$, then $\Ext_B^t(F, F)=k, \Ext_A^{c+d-t}(E, E)=k$, it follows that $\Ext_{\Lambda}^t(\begin{pmatrix}E\\ F\end{pmatrix}_{\Psi}, \begin{pmatrix}E\\ F\end{pmatrix}_{\Psi})=0$.

Thus
$$\Ext_{\Lambda}^t(\begin{pmatrix}E\\ F\end{pmatrix}_{\Psi}, \begin{pmatrix}E\\ F\end{pmatrix}_{\Psi})=\begin{cases}k  &  \text{if}\ t=0; \\ k & \text{if}\ t=c+d+1;\\0 &\text{otherwise.} \end{cases}$$

By Lemma \ref{iff1}, $\begin{pmatrix}E\\F\end{pmatrix}_{\Psi}$ is an exceptional $1$-cycle with
$S_\Lambda\begin{pmatrix}E\\ F \end{pmatrix}_{\Psi}\cong \begin{pmatrix}E\\ F \end{pmatrix}_{\Psi}[c+d+1]$.
\end{proof}

\begin{rem}
Note that Lemma \ref{1x1} is a special case of main result of \cite[Theorem 1.0.1]{Ba}, in this case, we give a short different proof.
\end{rem}

\begin{thm}  \label{product}
Let $A$ and $B$ be Gorenstein algebras, $E_*=(E_1, E_2, \cdots, E_n)$ a perfect exceptional $n$-cycle with respect to $(c_1, c_2, \cdots, c_n)$ in $K^b(A\mbox{-}{\rm proj})$, $F_*=(F_1, F_2, \cdots, F_m)$ a perfect exceptional $m$-cycle with respect to $(d_1, d_2, \cdots, d_m)$ in $K^b(B\mbox{-}{\rm proj})$. Assume that $\Lambda=\begin{pmatrix}A & N\\ 0 & B \end{pmatrix}$ with $A$-$B$-bimodule $N=E_n\otimes_k D(F_1)$. Then $E_*\mathop{\boxtimes} F_*$ is an exceptional $(n+m-1)$-cycle with respect to $(c_1,\cdots,c_{n-1}, d_1, \cdots, d_{m-1}, c_n+d_m+1)$.
\end{thm}

\begin{proof}

According to Proposition \ref{S4}, the product $E_*\mathop{\boxtimes} F_*$ is $S_{\Lambda}$-period with respect to the sequence of positive integers $(c_1,\cdots,c_{n-1}, d_1, \cdots, d_{m-1}, c_n+d_m+1)$ in $K^b(\Lambda\mbox{-}{\rm proj})$. By Lemma \ref{iff1} and Lemma \ref{iff2}, it remains to compute the $\Ext$-spaces of these $\Lambda$-modules.

The case $n=1, m=1$ holds by Lemma \ref{1x1}.

Assume that $n\ge 2, m\ge 2$ (the other two cases: $\text{(i)}\ n=1, m\ge 2; \ \text{(ii)}\ n\ge 2, m=1$ are proved similarly).

Since $E_*=(E_1, E_2, \cdots, E_n)$ is a perfect exceptional $n$-cycle with respect to $(c_1, c_2, \cdots, c_n)$, hence one has
$$\Ext^t_A(E_1, E_j)\cong
\begin{cases}
k, & {\rm if}\ j=1,\ t=0; \\
k,  & {\rm if}\  j=2,\ t=c_1;\\
0, & {\rm otherwise.}
\end{cases}$$

By Lemma \ref{XZ}, $\Ext^t_{\Lambda}(\begin{pmatrix}E_1\\ 0\end{pmatrix}, \begin{pmatrix}E_j\\ 0\end{pmatrix})\cong \Ext^t_{A}(E_1, E_j)$ for $1\le j\le n-1$, $\Ext^t_{\Lambda}(\begin{pmatrix}E_1\\ 0\end{pmatrix}, \begin{pmatrix}E_n\\ F_1\end{pmatrix}_{\Psi})\cong \Ext^t_{A}(E_1, E_n)$ and
$\Ext^t_{\Lambda}(\begin{pmatrix}E_1\\ 0\end{pmatrix}, \begin{pmatrix}0\\ F_j\end{pmatrix})=0$ for $1\le j \le m$.

Due to Lemma \ref{iff2}, $E_*\mathop{\boxtimes} F_*$ is an exceptional $(n+m-1)$-cycle with respect to $(c_1,\cdots,c_{n-1}, d_1, \cdots$, $d_{m-1}, c_n+d_m+1)$.
\end{proof}

\begin{ex}
Consider the following three algebras $A=kQ_1/I_1,\ B=kQ_2/I_2,\ \Lambda=kQ_3/I_3$, where $Q_1, Q_2, Q_3$ are given below,
$I_1=\langle \alpha^3 \rangle,\ I_2=\langle \beta^3 \rangle,\ I_3=I_1\cup I_2$.

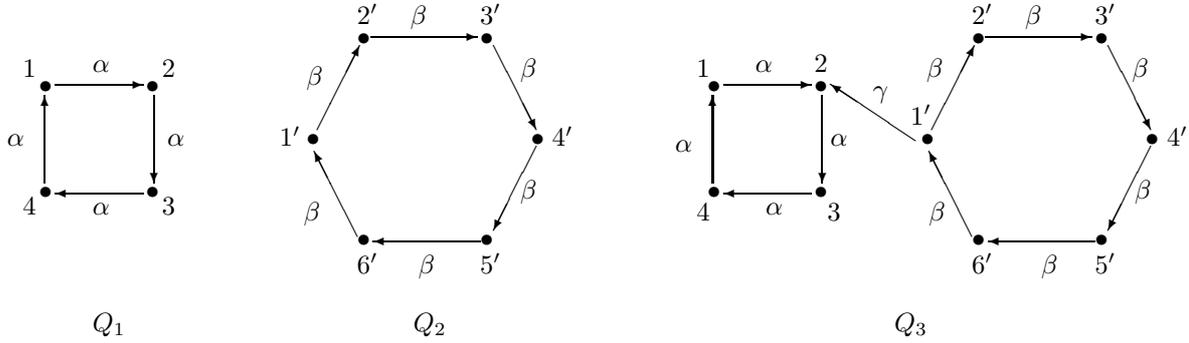
\begin{picture}(190,140)
\put(0,50){$\bullet$}   \put(0,90){$\bullet$}  \put(40,50){$\bullet$}  \put(40,90){$\bullet$}

\put(-6,44){$4$}   \put(-6,96){$1$}  \put(46,44){$3$}  \put(46,96){$2$}

\put(-12,70){$\alpha$}   \put(20,98){$\alpha$}  \put(48,70){$\alpha$}  \put(20,44){$\alpha$}

\put(20,0){$Q_1$}   \put(140,0){$Q_2$}    \put(320,0){$Q_3$}

\put(6,93){\vector(1,0){33}}   \put(43,89){\vector(0,-1){33}}     \put(39,52){\vector(-1,0){33}}   \put(2,56){\vector(0,1){33}}


\put(100,70){$\bullet$}  \put(119,108){$\bullet$}    \put(165,108){$\bullet$}  \put(184,70){$\bullet$}  \put(119,32){$\bullet$}  \put(165,32){$\bullet$}

\put(90,70){$1'$}  \put(119,116){$2'$}    \put(165,116){$3'$}  \put(192,70){$4'$}  \put(119,22){$6'$}  \put(165,22){$5'$}

\put(100,93){$\beta$}  \put(139,116){$\beta$}    \put(180,96){$\beta$}  \put(180,50){$\beta$}  \put(99,42){$\beta$}  \put(142,22){$\beta$}

\put(104,76){\vector(1,2){16}}  \put(124,111){\vector(1,0){40}}    \put(170,108){\vector(1,-2){16}}   \put(186,70){\vector(-1,-2){16}}
\put(165,34){\vector(-1,0){40}}   \put(119,36){\vector(-1,2){16}}


\put(250,50){$\bullet$}   \put(250,90){$\bullet$}  \put(290,50){$\bullet$}  \put(290,90){$\bullet$}

\put(246,42){$4$}   \put(246,96){$1$}  \put(295,42){$3$}  \put(290,98){$2$}

\put(238,68){$\alpha$}   \put(268,98){$\alpha$}  \put(272,44){$\alpha$}  \put(296,70){$\alpha$}

\put(256,93){\vector(1,0){33}}   \put(293,89){\vector(0,-1){33}}     \put(289,52){\vector(-1,0){33}}   \put(252,56){\vector(0,1){33}}

\put(330,70){$\bullet$}  \put(349,108){$\bullet$}    \put(395,108){$\bullet$}  \put(414,70){$\bullet$}  \put(349,32){$\bullet$}  \put(395,32){$\bullet$}

\put(326,78){$1'$}  \put(349,116){$2'$}    \put(395,116){$3'$}  \put(422,70){$4'$}  \put(349,22){$6'$}  \put(395,22){$5'$}

\put(332,96){$\beta$}  \put(369,116){$\beta$}    \put(409,96){$\beta$}  \put(410,50){$\beta$}  \put(333,42){$\beta$}  \put(375,22){$\beta$}

\put(312,88){$\gamma$}

\put(334,76){\vector(1,2){16}}  \put(354,111){\vector(1,0){40}}    \put(400,108){\vector(1,-2){16}}   \put(416,70){\vector(-1,-2){16}}
\put(395,34){\vector(-1,0){40}}   \put(349,36){\vector(-1,2){16}}    \put(328,72){\vector(-3,2){32}}
\end{picture}

\vskip15pt

Note that $(P(4), P(2))$ is an exceptional $2$-cycle in $K^b(A\mbox{-}{\rm proj})$, $(P(5'), P(3'), P(1'))$ an exceptional $3$-cycle in $K^b(B\mbox{-}{\rm proj})$, where $P(i)$ is the indecomposable projective modules corresponding to the vertex $i$.
It is clear that $\Lambda$ is the upper triangular matrix $\begin{pmatrix}A&  P(2)\otimes_kD(P(5'))\\0 & B\end{pmatrix}$, and $(P(4), P(5'), P(3'), P(1'))$ is an exceptional $(2+3-1)$-cycle in $K^b(\Lambda\mbox{-}{\rm proj})$.
However, algebra $\Lambda$ is a string algebra $($but not gentle$)$. So this is a new exceptional cycle that we do not know before.
\end{ex}

As a consequence we obtain the following description of (co)extension exceptional cycle of a given perfect exceptional cycle. Compare this result with
\cite[Proposition 3.5]{G}. Notice that $(k, k)$ is the unique exceptional cycle in $D^b(k)$ up to shift at each position.

\begin{cor}
Let $A$ be a Gorenstein algebra, $E_*=(E_1, E_2, \cdots, E_n)$ a perfect exceptional $n$-cycle in $K^b(A\mbox{-}{\rm proj})$.

$(1)$ If $\Lambda=\begin{pmatrix}A & E_n\\ 0 & k \end{pmatrix}$, then $E_*\boxtimes (k,k)$ is an exceptional $(n+1)$-cycle in $K^b(\Lambda\mbox{-}{\rm proj})$. In this case, $E_*\boxtimes (k,k)$ is called the extension of $E_*$.

\vskip5pt

$(2)$ If $\Lambda=\begin{pmatrix}k &  D(E_1)\\ 0 & A \end{pmatrix}$, then $(k,k)\boxtimes E_*$ is an exceptional $(n+1)$-cycle in $K^b(\Lambda\mbox{-}{\rm proj})$. In this case, $(k,k)\boxtimes  E_*$  is called the coextension of $E_*$.
\end{cor}


\begin{thebibliography}{99}

\bibitem[Add]{Add} N. Addington, New derived symmetries of some hyperk\"ahler varieties, Algebraic Geometry 3(2)(2016), 223-260.

\bibitem[AA]{AA} N. Addinton, P. S. Aspinwall, Categories of massless D-branes and del Pezzo surfaces, J. High Energy Phys. 07(2013)176.

\bibitem[An]{An} R. Anno, Spherical functor, arXiv: 0711.4409.

\bibitem[AL]{AL} R. Anno, T. Logvinenko, Spherical DG-functor, J. Eur. Math. Soc. 19(2017), 2577-2656.

\bibitem[ARS]{ARS} M. Auslander, I. Reiten, S. O. Smal$\o$, Representation Theory of Artin Algebras, Cambridge Studies in Adv. Math. 36., Cambridge
                   Univ. Press, 1995.

\bibitem[Ba]{Ba} F. Barbacovi, On the composition of two spherical twists, arXiv: 2006.06016v3 [math. AG].

\bibitem[Boc]{Boc} R. Bocklandt, Graded Calabi Yau algebras of  dimension
$3$, with an appendix ``The signs of Serre functor" by M. Van den
Bergh, J. Pure Appl. Algebra 212(1)(2008), 14-32.

\bibitem[Bob]{Bob} G. Bobi\'nski, Derived equivalence classification of the gentle two-cycle algebras, Algebr. Represent. Theory 20 (4)(2017), 857-869.

\bibitem[BK]{BK} A. Bondal, M. Kapranov, Representable functors, Serre functors, and mutations, Math. USSR Izv. 35(1990), 519-541.

\bibitem[BPP]{BPP} N. Broomhead, D. Pauksztello, D. Ploog, Discrete derived categories I: homomorphisms, autoequivalences and t-structures, Math. Z. 285(2017), 39-89.

\bibitem[C]{C} X. W. Chen, Singularities categories, Schur functors and traingular matrix rings, Algebr. Represent. Theory. 12(2009), 181-191.

\bibitem[CZ]{CZ} C. Cibils,  P. Zhang,  Calabi-Yau objects in triangulated categories, Trans. Amer. Math. Soc. 361(12)(2009), 6501-6519.

\bibitem[D]{D} W. Donovan, Grassmannian twists on the derived category via spherical functors, Proc. Lond. Math. Soc. (3)107(2013), 1053-1090.

\bibitem[DKSS]{DKSS} T. Dyckerhoff, M. Kapranov, V. Schechtman, Y. Soibelman, Spherical adjunctions of stable $\infty$-categories and the relative S-construction, arXiv: 2106.02873 [math.AG].

\bibitem[E]{E} A. Efimov, The M${\rm \ddot{o}}$bius contest paper: Braid group actions on triangulated categories, (2007).

\bibitem[G]{G} C. Geiss, Derived tame algebras and Euler-forms. With an appendix by the author and B. Keller, Math. Z. 239(4)(2002), 829-862.

\bibitem[GZ1]{GZ1} P. Guo, P. Zhang, Exceptional cycles in the bounded derived categories of quivers, Acta Math. Sinica (English Series) 36(3)(2020), 207-223.
\bibitem[GZ2]{GZ2} P. Guo, P. Zhang, Exceptional cycles for perfect complexes over gentle algebras, J. Algebra 565(2021), 160-195.

\bibitem[GZ3]{GZ3} P. Guo, P. Zhang, Exceptional cycles arising from nilpotent representations, preprint

\bibitem[HLS]{HLS} D. Halpern-Leistner, I. Shipman, Autoequivalences of derived categories via geometric invariant theory, Adv. Math. 303(2016), 1264-1299.

\bibitem[Ha]{Ha} D. Happel, On Gorenstein algebras, in Representation Theory of Finite Groups and Finite-Dimensional Algebra, Progr. Math. 95, Birkha\"user, Basel, 1991, 389-404.

\bibitem[HKP]{HKP} A. Hochenegger, M. Kalck, D. Ploog, Spherical subcategories in algebraic geometry, Math. Nachr. 289(11-12)(2016), 1450-1465.

\bibitem[Hu]{Hu} D. Huybrechts, Fourier-Mukai transforms in algebraic geometry, Oxford Maht. Monogr., Oxford Univ. Press, 2006.

\bibitem[KS]{KS} M. Kapranov, V. Schechtman, Perverse Schobers, arXiv: 1411. 2722v2 [math.AG].

\bibitem[Ko]{Ko} M. Kontsevich, ``Homological algebra of mirror symmetry'' in Proceedings of the International Congress of Mathematicians (Z\"{u}rich, 1994), Vol. 1,2, Birkh\"{a}user, Basel, 1995, 120-139.

\bibitem[Ku]{Ku} A. Kuznetsov, Calabi-Yau and fractional Calabi-Yau categories, J. Reine Angew. Math. 753(2019), 239-267.

\bibitem [M]{M} H. Meltzer, Tubular mutations, Colloq. Math. 74(1997), 267-274.

\bibitem[RV]{RV} I. Reiten, M. Van den Bergh, Noether hereditary abelian categories satisfying Serre functor, J. Amer. Math. Soc. 15(2)(2002), 295-366.

\bibitem[R]{R} R. Rouquiver, Categorification of $sl_2$ and braid groups, in Trends in representation theory of algebras and related topics, Contemp. Math. 406, Providence, RI, 2006, 137-167.

\bibitem[Seg]{Seg} E. Segal, All autoequivalences are spherical twists, Int. Math. Res. Not. 10(2018), 3137-3154.

\bibitem[ST]{ST} P. Seidel, R. Thomas, Braid group actions on derived categories of coherent sheaves, Duke Math. J. 108 (1)(2001), 37-108.

\bibitem[V]{V} Y. Volkov, Group generated by two twists along spherical sequences, arXiv: 1901.10904v2 [math.RT].

\bibitem[XZ]{XZ} B. L. Xiong, P. Zhang, Gorenstein-projective modules over trianglar matrix artin algebras, J. Algebra Appl. 11(4)(2012), 1250066.

\end{thebibliography}
\end{document}